\documentclass[12pt,twoside]{article}

\newtheorem{definition}{Definition}
\newtheorem{theorem}{Theorem}
\newtheorem{lemma}{Lemma}
\newtheorem{corollary}{Corollary}

\newtheorem{example}{Example}
\usepackage{amsmath}
\usepackage{amssymb}

\begin{document}
\date{}
\title{Some Spectral Problems for First Order Normal Differential Operators in the Weighted Hilbert Spaces of Vector-Functions
\vspace{60px}
\large\textbf{}\\
}

\author{\small{{\bf Zameddin I. Ismailov$^{1}$, Pembe Ipek Al$^{2}$, Mohammad Sababheh$^{3,4}$ }
\let\thefootnote\relax\footnote{{$^{1}$Karadeniz Technical University, Department of Mathematics, 61080, Trabzon, Türkiye.
\newline E-mail : \texttt{zameddin.ismailov@gmail.com} }}}
\let\thefootnote\relax\footnote{{$^{2}$Karadeniz Technical University, Department of Mathematics, 61080, Trabzon, Türkiye.
\newline E-mail : \texttt{ipekpembe@gmail.com} }}
\let\thefootnote\relax\footnote{{$^{3}$Department of Mathematics, Abdullah Al Salem University, Kuwait,\newline $^{4}$Princess Sumaya University for Technology, Department of Basic Sciences, Amman, Jordan
\newline E-mail(Corresponding Author) : \texttt{sababheh@yahoo.com} }}}

\date{}

\maketitle

\vspace{0.6cm}

\noindent
\small{\textbf{
\begin{center}
Abstract
\end{center}
}}
In this article, in order to the minimal operator generated by the first-order differential-operator expression in the weighted Hilbert space of vector functions in the finite interval to be formal normal, the relationship between the variable operator coefficient of this differential-operator expression and the weight function is established. Afterwards, the general form of all normal extension of the minimal operator is found using the Glazman-Krein-Naimark Method. Then, the structure of spectrum of such extensions is investigated. Later on, the issue of belonging to Schatten von-Neumann classes is explored, as well as the asymptotic behaviour of the singular numbers of the inverse of such normal extensions. Lastly, an approach is developed on all normal extension expressed in the weighted Hilbert spaces.
\vspace{0.125cm}
\footnotesize
\noindent
\hspace{0,38cm}

\begin{tabbing}
$\mathit{Key Words:}$ \textup {Differential operator, formal normal and normal operator, minimal and maximal } \\
\ \ \ \ \ \ \ \ \ \ \ \ \ \ \ \ \ \textup {operators, extension, Schatten von-Neumann operator class, $ s- $number of an} \\ \ \ \ \ \ \ \ \ \ \ \ \ \ \ \ \ \ \textup {operator, asyptotics of $ s- $number, weighted Hilbert space } \vspace*{0.5cm} 
\end{tabbing}

\noindent

\section{Introduction}
\label{Sec:1}

The Neumann Method refers to the broad representation theory proposed by Neumann in \cite{Neumann} for all selfadjoint extensions of a symmetric (formal selfadjoint) operator with linear closure and equal defect numbers given in a Hilbert space. The theory under consideration has undergone significant development since the 1960s, and \cite{Akhiezer, Everitt1997, Naimark, Zettl} describes the Glazman-Krein-Naimark Method. Later, in \cite{Gorbachuk, Rofe-Beketov}, the more advanced and practical Calkin-Gorbachuk Method was presented to science. The book \cite{Gorbachuk} discusses the application of the Calkin-Gorbachuk Method to two-point linear differential operators with operator coefficients in the Hilbert space of vector-functions. To gain further insight into this topic, we recommend reading the paper \cite{Rofe-Beketov}. 

\indent The Glazman-Krein-Naimark Method was used in \cite{Everitt} to express in terms of boundary values all selfadjoint extensions of the minimum operator generated by a class of first-order symmetric quasi-differential expressions in the scalar case and the structure of their spectrum was then analyzed.

\indent The discussion in \cite{Zettl} covered all selfadjoint extensions as well as some spectral analysis of the minimal operator generated by a linear classical formal selfadjoint differential expression of any order and smoothness coefficients in the weighted Hilbert space. Although a differential expression of any order can always be transformed into a first-order differential expression with matrix coefficients, it is more logical to examine problems suitable for first-order quasi-differential expressions in the weighted Hilbert space of vector-functions instead of such differential expressions. Furthermore, the Glazman-Krein-Naimark Method was used in \cite{Hao} to express all selfadjoint extensions of any order symmetric differential expression with complex coefficients, equal, and any number of defects in the weighted Hilbert space $ L_{\omega}^{2}(a,b) \ (\omega >0, \ a,b\in \overline{\mathbb{R}}) $. For linear closed formal normal operators, Kilpi \cite{Kilpi1, Kilpi2, Kilpi3} and Davis \cite{Davis} extended this theory in the 1950s. The Neumann type equation of all normal extensions of a linear closed unbounded formal normal operator given in a Hilbert space was later given by Biriuk and Coddington \cite{Biriuk, Coddington}. It has been noted, meanwhile, that applying this technique to differential operators presents significant challenges. The Coddington-Biriuk Method was unable to do away with the requirement to provide both the differential expression and the boundary value requirements that result in the differential operator in order to obtain a differential operator in any Banach space. 

\indent Different mathematicians, particularly Otelbaev and Ismailov, explored in depth all normal extensions and spectrum analysis of a linear formal normal differential expression in a without-weighted Hilbert space of vector-functions in their works (see \cite{Biyarov1993, Biyarov2016, Ismailov2003, Ismailov2005, Ismailov2006, Ismailov2012a, Ismailov2012b, Ismailov2011, Ismailov2012c, Ismailov2014a, Ismailov2014b}). The Calkin-Gorbachuk Method was typically applied in these investigations; however, only situations without a weight function were looked at. 

\indent In \cite{Coddington}, a densely-defined closed operator $ N $ in any Hilbert space is called a formally normal if $ D(N)\subset D(N^{*}) $ and $ \Vert N u \Vert = \Vert N^{*}u \Vert $ for all $ u \in D(N), $  where $ N^{*} $ is the Hilbert adjoint to the operator $ N. $ If a formally normal operator has no formally normal extension, then it is called maximal formally normal operator. If a formally normal operator $ N $ satisfies the condition $ D(N)= D(N^{*}) $, then it is called a normal operator.

\indent In the present study, in Section \ref{Sec:2}, the problem to be investigated is introduced and some results that will need later are given. In Section \ref{Sec:3}, in order to the minimal operator generated by the first-order differential-operator expression in the weighted Hilbert space of vector functions in the finite interval to be formal normal, the relationship between the variable operator coefficient of this differential-operator expression and the weight function is established. In Section \ref{Sec:4}, under the obtained conditions the general form of all normal extension of the minimal operator is found using the Glazman-Krein-Naimark Method. In Section \ref{Sec:5}, the structure of spectrum of such extensions is investigated. In Section \ref{Sec:6}, the issue of belonging to Schatten von-Neumann classes is explored, as well as the asymptotic behaviour of the singular numbers of the inverse of such normal extensions.  In Section \ref{Sec:7}, an approach is developed on all normal extension expressed in the weighted Hilbert spaces.

\section{Preliminaries}
\label{Sec:2}

Let $ H $ be a separable Hilbert space and $ \alpha : [0,1]\rightarrow [0,\infty ), \ \alpha \in C^{(2)}[0,1], \ \lambda\left( \left\lbrace t\in [0,1]: \alpha (t)=0\right) \right\rbrace =0,$ where $ \lambda $ is Lebesque measure in $ \mathbb{R}. $ In the weighted Hilbert space $ L_{\alpha}^{2}:=L_{\alpha}^{2}(H,(0,1)) $ of vector functions on $ (0,1), $ consider the following linear differential-operator expression in the form
$$
l(u)=u'(t)+A(t)u(t),
$$
where:\\
(1) for each $ t\in [0,1] $, $ A(t):D(A(t)) \subset H \rightarrow H $ is densely-defined closed operator;\\ 
(2) $ D(A(t))=D(A) $ doesn't depend on $ t\in [0,1] $ and is dense in $ H $; \\
(3) vector-functions $ A(t)f $ for any $ f\in D(A) $ is strongly continuously differentiable in $ H $ on $ [0,1] $; \\
(4) $ A_{r}(t)\geq \frac{\alpha'(t)}{2\alpha(t)} \ a. \ e. \ t\in (0,1). $

The formally adjoint expression of  $ l( ~^{.}~) $ in $ L^{2}_{\alpha} $ is in the form
\begin{eqnarray}
l^+(v) =  -\frac{(\alpha v)'(t)}{\alpha (t)}+A^{*}(t)v(t) \ a. \ e. \ t\in (0,1) \nonumber
\end{eqnarray}
Let define operator $ L_{0}' $ on the dense $ L^{2}_{\alpha} $ set of vector-functions $ D_{0}', $
$$
D_{0}':=\left\lbrace u(t)\in  L^{2}_{\alpha}: u(t)= \sum\limits_{k=1}^{n}\varphi_{k}(t)f_{k}, \ \varphi_{k}(t)\in C_{0}^{\infty}(0,1), \ f_{k}\in D(A), \ k=1,2,\ldots, n,\  n\in \mathbb{N} \right\rbrace 
$$
as
$$
L_{0}'u=l(u), \ u\in D_{0}'.
$$

The operator $ L_{0}' $ has a closure in $ L_{\alpha}^{2}. $ Indeed, we have
\begin{eqnarray}
2Re\langle L_{0}'u,u \rangle_{L_{\alpha}^{2}} & = & \langle L_{0}'u,u \rangle_{L_{\alpha}^{2}} + \langle u,L_{0}'u \rangle_{L_{\alpha}^{2}} \nonumber \\
& = & \langle u'(t)+A(t)u(t),u(t) \rangle_{L_{\alpha}^{2}} + \langle u,u'(t)+A(t)u(t) \rangle_{L_{\alpha}^{2}} \nonumber \\
& = & \langle u'(t),u(t) \rangle_{L_{\alpha}^{2}} + \langle u(t),u'(t)  \rangle_{L_{\alpha}^{2}} + \langle (A(t)+A^{*}(t))u(t),u(t) \rangle_{L_{\alpha}^{2}} \nonumber \\
& = & \int\limits_{0}^{1} \langle u'(t),\alpha (t)u(t) \rangle_{H}dt+ \int\limits_{0}^{1} \langle u(t),u'(t) \rangle_{H}\alpha (t)dt + \langle (2A_{r}(t)u(t),u(t) \rangle_{L_{\alpha}^{2}} \nonumber \\
& = & -\int\limits_{0}^{1} \langle u(t),\alpha' (t)u(t) \rangle_{H}dt + \langle (2A_{r}(t)u(t),u(t) \rangle_{L_{\alpha}^{2}} \nonumber \\
& = & -\int\limits_{0}^{1} \langle \frac{\alpha' (t)}{\alpha(t)}u(t),u(t) \rangle_{H}\alpha(t)dt + \langle (2A_{r}(t)u(t),u(t) \rangle_{L_{\alpha}^{2}} \nonumber \\
& = &  \langle \left( 2A_{r}(t)-\frac{\alpha' (t)}{\alpha(t)}\right) u(t),u(t) \rangle_{L_{\alpha}^{2}} \ a. \ e. \ t\in (0,1). \nonumber 
\end{eqnarray}
Since the operator $ A_{r}(t)\geq \frac{\alpha'(t)}{2\alpha(t)} \ a. \ e. \ t\in (0,1), $ then $ L_{0}' $ is accretive, that is $ Re \langle L_{0}'u, u \rangle_{L_{\alpha}^{2}}\geq 0, \ u\in D_{0}' .  $
Consequently, the operator $ L_{0}' $ has a closure in $ L_{\alpha}^{2}. $ The closure of $ L_{0}' $ in $ L_{\alpha}^{2} $ is called the minimal operator generated by differential-operator expression $ l( \cdot ) $ and is denoted by $ L_{0}. $

\indent In a similar way the minimal operator $ L_{0}^{+} $ corresponding to differential-operator expression $ l^{+}(\cdot ) $ in $ L_{\alpha}^{2} $ can be constructed. The operator $ L(L^{+}) $ adjoint of $ L_{0}^{+}(L_{0}) $ is called the maximal operator corresponding to $ l(\cdot )(l^{+}(\cdot ))  $ in $ L_{\alpha}^{2} $ \cite{Hormander,Ismailov2006}. That is, $ L=(L_{0}^{+})^{*}, \ L^{+}=L_{0}^{*} $ and $ L_{0}\subset L, \  L_{0}^{+}\subset L^{+} $. In this case
$
D(L)=W_{2,\alpha}^{1}(D(A),(0,1)) \ \text{and} \ D(L_{0})=\stackrel{_o}{W}_{2,\alpha}^{1}(D(A),(0,1)).
$ 

Note that the minimal operator $ L_{0} $ is not maximal. Indeed, if $ A $ is a normal operator, the differential operator expression $ l(\cdot) $ with the boundary condition 
$$
(\sqrt{\alpha}u)(1)=exp\left\lbrace -i\int\limits_{0}^{1}A_{i}(t)dt\right\rbrace (\sqrt{\alpha}u)(0) 
$$ 
generates a normal extension of $ L_{0} $ in $ L_{\alpha}^{2}.$ 

\indent The aim of this paper is to establish the relationship between the formal normality property of the minimal operator generated by the differential operator expression $ l(\cdot) $ and the operator coefficient $ A(\cdot) $ and the weighted function $ \alpha (\cdot) $, to  obtain the general form of all normal extensions of the minimal operator $ L_{0} $ in $ L_{\alpha}^{2}$ in terms of boundary values, to investigate the spectrum set of them, to explore the issue of belonging to Schatten von-Neumann classes as well as the asymptotic behaviour of the singular numbers of the inverse of such normal extensions, and to make a generalization.

\section{On normality of first order linear differential operator in weighted Hilbert space}
\label{Sec:3}
In this section, in order to the minimal operator  $ L_{0} $ to be formal normal, the relationship between the variable operator coefficient $ A(\cdot) $ and the weight function  $ \alpha (\cdot) $ will be established.

\begin{theorem}
\label{thm1}
If the minimal operator $ L_{0} $ is formally normal in $ L_{\alpha}^{2} $, then
\begin{eqnarray}
A^{*}(t)A(t)-A(t)A^{*}(t)=\left( 2A_{r}(t)-  \frac{\alpha'(t)}{\alpha(t)}I\right)' \ a. \ e. \ t\in (0,1) \nonumber
\end{eqnarray}
where $ A_{r}(t)=ReA(t)=\frac{1}{2} \overline{\left( A(t)+A^{*}(t)\right)}, \ t\in (0,1) $ and $ I $ is identity operator in $ H. $
\end{theorem}
\noindent {\it Proof.}
Let  $ L_{0} $ be formally normal in $ L_{\alpha}^{2} $. Then, for each $ u\in D(L_{0})  $  we have 
$$
\Vert L_{0} \Vert^{2}_{L_{\alpha}^{2}}=\Vert L_{0}^{*} \Vert^{2}_{L_{\alpha}^{2}},
$$
that is
$$
\Vert u'(t)+A(t)u(t) \Vert^{2}_{L_{\alpha}^{2}}=\Vert  \frac{(\alpha u)'(t)}{\alpha (t)}-A^{*}(t)u(t) \Vert^{2}_{L_{\alpha}^{2}} \ a. \ e. \ t\in (0,1).
$$
Then, for each $ u(t)=\varphi(t)f, \ \varphi (t)\in \stackrel{_o}{W}_{2}^{2}(0,1), \ f\in D(A) $, we have
\begin{eqnarray}
\langle \varphi '(t)f+A(t)\varphi(t)f, \varphi '(t)f+A(t)\varphi(t)f \rangle_{L_{\alpha}^{2}} = \langle \varphi '(t)f+\Delta(t)\varphi(t)f, \varphi'(t)f+\Delta(t)\varphi(t)f \rangle_{L_{\alpha}^{2}}, \nonumber
\end{eqnarray}
where $ \Delta(t)=\dfrac{\alpha'(t)}{\alpha (t)}I-A^{*}(t) \ a. \ e. \ t\in (0,1). $
Thus, we get
\begin{eqnarray}
& & \langle \varphi '(t)f, \varphi '(t)f \rangle_{L_{\alpha}^{2}}+\langle \varphi '(t)f, A(t)\varphi(t)f \rangle_{L_{\alpha}^{2}} 
 +  \langle A(t)\varphi (t)f, \varphi '(t)f \rangle_{L_{\alpha}^{2}}+\langle A(t)\varphi (t)f, A(t)\varphi (t)f \rangle_{L_{\alpha}^{2}} \nonumber \\
& = & \langle \varphi '(t)f, \varphi '(t)f \rangle_{L_{\alpha}^{2}}+\langle \varphi '(t)f, \Delta(t)\varphi (t)f \rangle_{L_{\alpha}^{2}}
 +  \langle \Delta(t)\varphi (t)f, \varphi '(t)f \rangle_{L_{\alpha}^{2}}+\langle \Delta(t)\varphi (t)f, \Delta(t)\varphi (t)f \rangle_{L_{\alpha}^{2}} \ a. \ e. \ t\in (0,1). \nonumber 
\end{eqnarray}
Hence, we have
\begin{eqnarray}
& & \langle \varphi '(t)f,(A(t)-\Delta(t)) \varphi (t)f \rangle_{L_{\alpha}^{2}}+\langle \varphi (t)f,(A^{*}(t)-\Delta^{*}(t)) \varphi '(t)f \rangle_{L_{\alpha}^{2}} + \langle \varphi (t)f,(A^{*}(t)A(t)-\Delta^{*}(t)\Delta(t)) \varphi (t)f \rangle_{L_{\alpha}^{2}} =0 \ a. \ e. \ t\in (0,1). \nonumber
\end{eqnarray}
Thus, we have
\begin{eqnarray}
& & \int\limits_{0}^{1} \varphi'(t)\varphi(t)\langle f, \left( A(t)-\Delta(t) \right)f\rangle_{H}\alpha(t)dt  + \int\limits_{0}^{1} \varphi(t)\varphi'(t)\langle f, \left( A^{*}(t)-\Delta^{*}(t) \right)f\rangle_{H}\alpha(t)dt \nonumber \\
& + & \int\limits_{0}^{1} \varphi(t)\varphi(t)\langle f, \left( A^{*}(t)A(t)-\Delta^{*}(t)\Delta(t) \right)f\rangle_{H}\alpha(t)dt =0. \nonumber
\end{eqnarray}
Then,
\begin{eqnarray}
& & \int\limits_{0}^{1} \alpha(t)\varphi(t)\varphi'(t)\langle f, \left( A(t)-\Delta(t) + A^{*}(t)-\Delta^{*}(t)\right)f\rangle_{H}dt + \int\limits_{0}^{1} \alpha(t)\varphi^{2}(t)\langle f, \left( A^{*}(t)A(t)-\Delta^{*}(t)\Delta(t) \right)f\rangle_{H}dt =0. \nonumber
\end{eqnarray}
Since for each $ t\in (0,1) $ the following relations are true
$$
A(t)-\Delta(t)=A(t)-\left( \frac{\alpha'(t)}{\alpha(t)} -A^{*}(t)\right) = 2A_{r}(t)-\frac{\alpha'(t)}{\alpha(t)} \ a. \ e. \ t\in (0,1),
$$
$$
A^{*}(t)-\Delta^{*}(t)=A^{*}(t)-\left( \frac{\alpha'(t)}{\alpha(t)} -A(t)\right) = 2A_{r}(t)-\frac{\alpha'(t)}{\alpha(t)} \ a. \ e. \ t\in (0,1),
$$
and
\begin{eqnarray}
\Delta^{*}(t)\Delta(t) & = & \left( \frac{\alpha'(t)}{\alpha(t)}-A(t)\right)\left( \frac{\alpha'(t)}{\alpha(t)}-A^{*}(t)\right)  \nonumber \\
& = & \left( \frac{\alpha'(t)}{\alpha(t)}\right)^{2}- \frac{\alpha'(t)}{\alpha(t)}A^{*}(t)- \frac{\alpha'(t)}{\alpha(t)}A(t)+A(t)A^{*}(t)  \nonumber \\
& = & \left( \frac{\alpha'(t)}{\alpha(t)}\right)^{2}- \frac{\alpha'(t)}{\alpha(t)}2A_{r}(t)+A(t)A^{*}(t) \ a. \ e. \ t\in (0,1), \nonumber
\end{eqnarray}
then we have
\begin{eqnarray}
& & 2\int\limits_{0}^{1} \alpha(t)\varphi(t)\varphi'(t)\langle f, \left(  2A_{r}(t)-\frac{\alpha'(t)}{\alpha(t)}\right)f\rangle_{H}dt  \nonumber \\
& + & \int\limits_{0}^{1} \alpha(t)\varphi^{2}(t)\langle f, \left( A^{*}(t)A(t)-A(t)A^{*}(t)+2\frac{\alpha'(t)}{\alpha(t)}A_{r}(t)-\left( \frac{\alpha'(t)}{\alpha(t)}\right)^{2}\right)f\rangle_{H}dt =0. \nonumber
\end{eqnarray}
From the last relation, for any real-valued functions $ \varphi(t), \psi(t) \in \stackrel{_o}{W}_{2}^{2}(0,1) $, we have
\begin{eqnarray}
2\int\limits_{0}^{1} \alpha(t)(\varphi(t)+\psi(t))(\varphi'(t)+\psi'(t))\Delta_{1}(t)dt + \int\limits_{0}^{1} \alpha(t)(\varphi^{2}(t)+2\varphi(t)\psi(t)+\psi^{2}(t))\Delta_{2}(t)dt =0, \nonumber
\end{eqnarray}
that is,
\begin{eqnarray}
\int\limits_{0}^{1} \alpha(t)(\varphi(t)\psi'(t)+\varphi'(t)\psi(t))\Delta_{1}(t)dt + \int\limits_{0}^{1} \alpha(t)\varphi(t)\psi(t)\Delta_{2}(t)dt =0, \nonumber
\end{eqnarray}
where $ \Delta_{1}(t)=\langle f, \left(  2A_{r}(t)-\frac{\alpha'(t)}{\alpha(t)}I\right)f\rangle_{H} \ a. \ e. \ t\in (0,1) $ and 
$$ \Delta_{2}(t)=\langle f, \left( A^{*}(t)A(t)-A(t)A^{*}(t)+2\frac{\alpha'(t)}{\alpha(t)}A_{r}(t)-\left( \frac{\alpha'(t)}{\alpha(t)}\right)^{2}I\right)f\rangle_{H} \ a. \ e. \ t\in (0,1). $$
Hence, we get
\begin{eqnarray}
\int\limits_{0}^{1} \left( \alpha(t)\varphi'(t)\Delta_{1}(t)+\alpha(t)\varphi(t)\Delta_{2}(t)\right) \psi(t)dt + \int\limits_{0}^{1} \left( \alpha(t)\varphi(t)\Delta_{1}(t)\right) \psi'(t)dt =0. \nonumber
\end{eqnarray}
Then from \cite{Shilov}, we have
$$
\alpha(t)\varphi'(t)\Delta_{1}(t)+\alpha(t)\varphi(t)\Delta_{2}(t)-\left(  \alpha(t)\varphi(t)\Delta_{1}(t)\right)'=0 \ a. \ e. \ t\in (0,1),
$$
that is,
$$
\alpha(t)\Delta_{2}(t)=(\alpha(t)\Delta_{1}(t))' \ a. \ e. \ t\in (0,1).
$$
Hence,
\begin{eqnarray}
& & \alpha(t)\langle f, \left( A^{*}(t)A(t)-A(t)A^{*}(t)+2\frac{\alpha'(t)}{\alpha(t)}A_{r}(t)-\left( \frac{\alpha'(t)}{\alpha(t)}\right)^{2}I \right)f\rangle_{H}\nonumber \\
& = & \alpha(t)\langle f, \left(2A_{r}(t)- \frac{\alpha'(t)}{\alpha(t)}\right)'f\rangle_{H}+ \alpha'(t)\langle f, \left(  2A_{r}(t)-\frac{\alpha'(t)}{\alpha(t)}I\right)f\rangle_{H} \ a. \ e. \ t\in (0,1). \nonumber
\end{eqnarray}
Then,
$$
\alpha (t) \left( A^{*}(t)A(t)-A(t)A^{*}(t)\right) +2\alpha'(t)A_{r}(t)- \frac{(\alpha'(t))^{2}}{\alpha(t)}=\alpha(t)\left( 2A_{r}'(t)- \frac{\alpha'(t)}{\alpha(t)}\right)'+2\alpha'(t)A_{r}-\frac{(\alpha'(t))^{2}}{\alpha (t)} \ a. \ e. \ t\in (0,1),
$$
that is
$$
\alpha (t) \left( A^{*}(t)A(t)-A(t)A^{*}(t)\right)=\alpha (t)\left( 2A_{r}(t) - \frac{\alpha'(t)}{\alpha(t)}I\right)' \ a. \ e. \ t\in (0,1).
$$
Consequently,
$$
A^{*}(t)A(t)-A(t)A^{*}(t)=\left( 2A_{r}(t)- \frac{\alpha'(t)}{\alpha(t)}I\right)' \ a. \ e. \ t\in (0,1).
$$
\begin{corollary}
\label{cor1}
Let for each $ t\in [0,1] $, $ A(t) $ be a linear bounded normal operator having strongly continuous derivatives in Hilbert space $ H $. If the minimal operator $ L_{0} $ is formally normal in $ L_{\alpha}^{2} $, then
\begin{eqnarray}
\label{equ1}
A_{r}(t)=\frac{\alpha'(t)}{2\alpha(t)}I+C \ a. \ e. \ t\in (0,1),
\end{eqnarray}
where $ I:H\rightarrow H $ is identity operator and $ C:D(A) \subset H\rightarrow H $ positive constant operator.
\end{corollary}

\section{Description of the normal extensions}
\label{Sec:4}

In this section, with the use of Glazman-Krein-Naimark method we will study the abstract representation of all normal extensions of the minimal operator $L_{0}$ in $ L^{2}_{\alpha}$. \\
Now, let $ A(t)=A_{r}(t)+iA_{i}(t), $ where \\
(1) $ A_{r}(t)=A_{r}^{*}(t) $ and $ A_{i}(t)=A_{i}^{*}(t) $ for each $ t\in [0,1] $; \\
(2) $ A_{r}(t)A_{i}(t)=A_{i}(t)A_{r}(t) $ for each $ t\in [0,1] $; \\
(3) vector-functions $ A_{r}(t)f $ and $ A_{i}(t)f $ for any $ f\in D(A) $ are strongly continuously differentiable in $ H $ on $ [0,1] $. \\

Let define the following unitary operator
$$
\mathcal{F}_{\alpha}: L^{2} \rightarrow  L^{2}_{\alpha}, \ \mathcal{F}_{\alpha}u(t)=\left( \frac{1}{\sqrt{\alpha}}u\right) (t) \ a. \ e. \ t\in (0,1),
$$
where $ L^{2}=L^{2}(H,(0,1))  $ (for the definition of unitary operator in different spaces see \cite{Kato}).
In this case,
$$
\mathcal{F}_{\alpha}^{-1}u(t): L^{2}_{\alpha} \rightarrow  L^{2}, \ \mathcal{F}_{\alpha}^{-1}u(t)= \left( \sqrt{\alpha}u \right) (t) \ a. \ e. \ t\in (0,1).
$$
Hence, from the condition of normality (\ref{equ1}),
\begin{eqnarray}
l(\mathcal{F}_{\alpha}u)(t) & = & (\mathcal{F}_{\alpha}u)'(t)+A(t)(\mathcal{F}_{\alpha}u)(t) \nonumber \\
& = & \left( \frac{1}{\sqrt{\alpha}}u\right)'(t)+A(t)\left( \frac{1}{\sqrt{\alpha}}u\right)(t) \nonumber \\
& = &  \frac{1}{\sqrt{\alpha (t)}}u'(t)- \frac{\alpha'(t)}{2\alpha(t)\sqrt{\alpha (t)}}u(t)+\frac{1}{\sqrt{\alpha (t)}}A(t)u(t) \nonumber  \\
& = & \frac{1}{\sqrt{\alpha (t)}}\left( u'(t)-\frac{\alpha '(t)}{2\alpha (t)}u(t)+A(t)u(t) \right) \nonumber \\
& = & \frac{1}{\sqrt{\alpha (t)}}\left( u'(t)+\left( A_{r}(t)-\frac{\alpha '(t)}{2\alpha (t)}\right) u(t)+iA_{i}(t)u(t) \right) \nonumber \\
& = & \frac{1}{\sqrt{\alpha (t)}}\left( u'(t)+Cu(t)+iA_{i}(t)u(t) \right) \ a. \ e. \ t\in (0,1). \nonumber 
\end{eqnarray}
Thus, we get 
$$
\sqrt{\alpha(t)}l(\mathcal{F}_{\alpha}u)(t)=u'(t)+Cu(t)+iA_{i}(t)u(t) \ a. \ e. \ t\in (0,1),
$$
that is,
$$
\left( \mathcal{F}_{\alpha}^{-1}l \mathcal{F}_{\alpha}\right) (u)= u'(t)+Cu(t)+iA_{i}(t)u(t), \ \mathcal{F}_{\alpha}^{-1}l \mathcal{F}_{\alpha}: L^{2}\rightarrow L^{2} \ a. \ e. \ t\in (0,1).
$$
Let define the following linear differential-operator expression in the form
$$
k(u)= u'(t)+Cu(t)+iA_{i}(t)u(t)
$$
in the Hilbert space $ L^{2} $ of vector functions on $ (0,1). $ 

\indent The minimal operator $ K_{0}(K_{0}^{+}) $ corresponding to differential-operator expression $ k(\cdot) (k^{+}(\cdot )) $ in $ L^{2} $ can be constructed by using same technique \cite{Hormander,Ismailov2006}. The operator $ K(K^{+}) $ adjoint of $ K_{0}^{+}(K_{0}) $ is called the maximal operator corresponding to $ k(\cdot )(k^{+}(\cdot ))  $ in $ L^{2} $ \cite{Hormander,Ismailov2006}. That is, $ K=(K_{0}^{+})^{*}, \ K^{+}=K_{0}^{*} $ and $ K_{0}\subset K, \  K_{0}^{+}\subset K^{+} $. In this case
$
D(K)=W_{2}^{1}(D(A),(0,1)) \ \text{and} \ D(K_{0})=\stackrel{_o}{W}_{2}^{1}(D(A),(0,1)).
$ 

In this case, $ \mathcal{F}_{\alpha}^{-1}L_{0}\mathcal{F}_{\alpha}=K_{0}, \ \mathcal{F}_{\alpha}^{-1}L\mathcal{F}_{\alpha}=K $ and is $ L_{0}\subset \widetilde{L}\subset L, $ then $ K_{0}\subset \mathcal{F}_{\alpha}^{-1}\widetilde{L}\mathcal{F}_{\alpha}\subset K. $ Since $ \mathcal{F}_{\alpha}:L^{2}\rightarrow L^{2}_{\alpha} $ is a unitary operator, then for the normality of extension $ \widetilde{L}: L_{0}\subset \widetilde{L}\subset L  $ the necessary and sufficient condition is the normality of extension  $ \widetilde{K}: K_{0}\subset \widetilde{K}=\mathcal{F}_{\alpha}^{-1}\widetilde{L}\mathcal{F}_{\alpha}\subset K $.

In a similar manner, one can construct the minimal $ M_{0} $ and the maximal operator $ M $ corresponding to the differential-operator expression
$$
m(u)=u'(t)+Cu(t).
$$
Hilbert space $ L^{2} $ of vector-functions (see \cite{Ismailov2006}).

Now, let  $ U(t,s), \ t,s \in [0,1], $ be the family of evolution operators corresponding to the homogeneous differential equation 
$$
\begin{cases}
      U_{t}'(t,s)f+iA_{i}(t)U(t,s)f=0, & t,s\in [0,1],\\
      U(s,s)f=f, & f\in D(A).
\end{cases}
$$
The operator $ U(t,s), \ t,s \in [0,1] $ is linear  unitary operator in $ H $ and 
$$
U^{*}(t,s)= U(s,t), \  U^{-1}(t,s)= U(s,t)
$$
(for more detail analysis of this operator see \cite{Goldstein}, \cite{Krein} and \cite{Bruk}). Let introduce the operator 
$$
Uv(t):= U(t,0)v(t), \ U: L^{2}\rightarrow L^{2}.
$$
In this case, it is easy to see that for the differentiable vector-function $ v(t)\in L^{2} $ with $ v(t)\in D(A), t \in [0,1], $ is valid the following relation:
\begin{eqnarray}
k(Uv) & = & (Uv)'(t)+CUv(t)+iA_{i}(t)Uv(t) \nonumber \\
& = & U'v(t)+Uv'(t)+CUv(t)+iA_{i}(t)Uv(t) \nonumber \\
& = & U ( v'(t)+Cv(t)) + \left( U'+iA_{i}(t)U \right)v(t) \nonumber \\
& = & Um(v). \nonumber
\end{eqnarray}
From this $ U^{-1}kU(v)= m(v).$ Hence, it is clear that if the operator $ \widetilde{K} $ is some extension of the minimal operator $ K_{0}, $ that is $ K_{0}\subset \widetilde{K}\subset K, $ then
$$
U^{-1}K_{0} U=M_{0}, \ M_{0}\subset U^{-1}\widetilde{K}U=\widetilde{M}\subset M, \ U^{-1}K U=M.
$$
For example, we will prove the validity of the last relation. It is known that
$$
D(M)=\left\lbrace u(t) \in L^{2}: \ u(t) \ absolutey \ continuous \ on \ (0,1) \ and \ m(u)\in L^{2} \right\rbrace 
$$
and 
$$
D(M_{0})=\left\lbrace u(t)\in D(M): u(0)=u(1)=0 \right\rbrace.
$$
If $ v(t)\in D(M), $ then $ Uv(t) $ absolutely continuous on $ (0,1) $ that is, $ Uv(t)\in D(K). $ Furthermore, from the last relation $ M \subset U^{-1}K U.$ Contrary, if a some vector-function $ z(t)\in D(K), $ then the element $ U^{-1}z(t) $ absolutely continuous on the interval $ (0,1) $ and
\begin{eqnarray}
m(U^{-1}z(t)) & = & (U^{-1}z(t))'+C(U^{-1}z(t)) \nonumber \\
& = & U^{-1} (z'(t)+Cz(t)+iA_{i}z(t)) \nonumber \\
& = & U^{-1}s(z(t))\in L^{2}, \nonumber
\end{eqnarray}
that is, $ U^{-1}z(t)\in D(M), $ and from the last relation $ U^{-1}K\subset MU^{-1}, $ that is, $ U^{-1}KU\subset M. $ Hence, $ U^{-1}KU=M. $ Therefore, the operator $ U $ is a one to one map of $ D(M) $ onto $ D(K). $

From \cite{Ismailov2006}, each normal extensions of the minimum operators $ M_{0} $ and $ K_{0} $ in $ L^{2} $  are expressed by Theorems \ref{thm2} and \ref{thm3}, respectively.

\begin{theorem}
\label{thm2}
Let $ C^{1/2}\left[ D(M)\cap D(M^{+}) \right]\subset W_{2}^{1}(H,(0,1)).  $ Each normal extension $ \widetilde{M} $ of the minimal operator $ M_{0} $ in $ L^{2} $ is generated by the differential-operator expression $ m(\cdot ) $ with the boundary condition
\begin{equation}
\label{eq10}
u(1)=Wu(0),
\end{equation}
where $ W $  is unitary operator in $ H $ and $ CW=WC. $  The unitary operator $ W $ is determined uniquely by the extension $ \widetilde{M} $, i.e.,  $ \widetilde{M}=M_{W}. $  

\indent On the contrary, the restriction of the maximal operator $ M $ to the manifold of vector-functions $ u(t)\in D(M)\cap D(M^{+}) $ that satisfy the condition (\ref{eq10}) for some unitary operator $ W, $ where $ CW=WC $, is a normal extension of the minimal operator in the space $ L^{2} .$
\end{theorem}

\begin{theorem}
\label{thm3}
Let $ C^{1/2}\left[ D(K)\cap D(K^{+}) \right]\subset W_{2}^{1}(H,(0,1)).  $ Each normal extension $ \widetilde{K} $ of the minimal operator $ K_{0} $ in $ L^{2} $ is generated by the differential-operator expression $ k(\cdot ) $ with the boundary condition
\begin{equation}
\label{eq11}
u(1)=U(1,0)Wu(0),
\end{equation}
where $ W $  is unitary operator in $ H $ and $ CW=WC. $ The unitary operator $ W $ is determined uniquely by the extension $ \widetilde{K} $, i.e.,  $ \widetilde{K}=K_{W}. $  

\indent On the contrary, the restriction of the maximal operator $ K $ to the manifold of vector-functions $ u(t)\in D(K)\cap D(K^{+}) $ that satisfy the condition (\ref{eq11}) for some unitary operator $ W, $ where $ CW=WC $, is a normal extension of the minimal operator in the space $ L^{2} .$
\end{theorem}

Hence, from the relation between the normal extension of the minimal operator $  L_{0} $ in $ L^{2}_{\alpha} $  and the normal extension of the minimal operator $  K_{0} $ in $ L^{2} $, all normal extensions of the minimal operator $  L_{0} $ in $ L^{2}_{\alpha} $ can be given by the following theorem.

\begin{theorem}
\label{thm4}
Let $ C^{1/2}\left[ D(L)\cap D(L^{+}) \right]\subset W_{2,\alpha}^{1}(H,(0,1)).  $ Each normal extension $ \widetilde{L} $ of the minimal operator $ L_{0} $ in $ L^{2}_{\alpha} $ is generated by the differential-operator expression $ l(\cdot ) $ with the boundary condition
\begin{equation}
\label{eq12}
(\sqrt{\alpha}u)(1)=U(1,0)W(\sqrt{\alpha}u)(0),
\end{equation}
where $ W $  is unitary operator in $ H $ and $ CW=WC. $  The unitary operator $ W $ is determined uniquely by the extension $ \widetilde{L} $, i.e.,  $ \widetilde{L}=L_{W}. $  

\indent On the contrary, the restriction of the maximal operator $ L $ to the manifold of vector-functions $ u(t)\in D(L)\cap D(L^{+}) $ that satisfy the condition (\ref{eq12}) for some unitary operator $ W, $ where $ CW=WC $, is a normal extension of the minimal operator in the space $ L^{2}_{\alpha} .$
\end{theorem}

\section{Spectrum of the normal extensions}
\label{Sec:5}

In this section, we will investigate the geometry of the spectrum of the normal extension $ L_{W} $ of the minimal operator $ L_{0} $ in $ L^{2}_{\alpha}. $
\begin{theorem}
\label{thm5}
If $ L_{W} $ is a normal extension of minimal operator $ L_{0} $ and $ M_{W}=U^{-1}\mathcal{F}_{\alpha}^{-1}L_{W}\mathcal{F}_{\alpha} U$ corresponds to the normal extension of a minimal operator $ M_{0}, $ then the spectrum of these extensions are given by 
$$
\sigma (L_{W})= \sigma (M_{W}).
$$
\end{theorem}
\noindent {\it Proof.}
Consider a problem to the spectrum for a normal extension $ L_{W} $ of minimal operator $ L_{0} $, that is 
$$
L_{W}u=\lambda u+f, \ \lambda \in \mathbb{C}, \ f \in L^{2}_{\alpha}.
$$
Hence, we have $ (L_{W}-\lambda I)u=f $ or $ (\mathcal{F}_{\alpha}U M_{W}U^{-1}\mathcal{F}_{\alpha}^{-1}-\lambda I)u=f. $ Thus, $ \mathcal{F}_{\alpha}U(M_{W}-\lambda I )(U^{-1}\mathcal{F}_{\alpha}^{-1}u)=f. $ Therefore, the validity of the theorem is clear.
\begin{theorem}
\label{thm6}
The spectrum of the normal extensions $ L_{W} $ is of form
$$
\sigma (L_{W})=\left\lbrace \lambda\in \mathbb{C}: \lambda = ln\vert \mu \vert^{-1}+i(2n\pi-arg\mu ), \ n\in \mathbb{Z}, \ \mu\in \sigma\left( W^{*}e^{-C}  \right) \right\rbrace .
$$
\end{theorem}
\noindent {\it Proof.}
Since $ \sigma (L_{W})= \sigma (M_{W}), $ then we investigate the spectrum of normal extension $ M_{W} $ in $ L^{2}. $ Let consider the spectrum for the normal extension $ M_{W} $, that is,
$$
u'(t)+Cu(t)=\lambda u(t)+f(t), \ \lambda\in \mathbb{C}, \ u, \ f\in L^{2}
$$
with the conditions (\ref{eq10}), where $ W $ is unitary operators in $ H $ and $ CW=WC. $ \\

The general $ L^{2}-$solution of this differential equation is as follows: 
$$
u(t;\lambda)=e^{-(C-\lambda )t}f_{\lambda}+\int\limits_{0}^{t}e^{-(C-\lambda )(t-s)}f(s)ds, \ f_{\lambda}\in H.
$$
From the boundary condition (\ref{eq10}), we have
$$
\left(W^{*} e^{-C}- e^{-\lambda}  \right) f_{\lambda}=W^{*} e^{-\lambda}\int\limits_{0}^{1}e^{-(C-\lambda )(1-s)}f(s)ds.
$$
One can see that the necessary and sufficient condition for $ \lambda \in \sigma (M_{W}) $ is 
$$
e^{-\lambda}=\mu\in \sigma\left( W^{*}e^{-C}  \right) .
$$
Therefore,
$$
\lambda = ln\vert \mu \vert^{-1}+i(2n\pi-arg\mu ), \ n\in \mathbb{Z}, \ \mu\in \sigma\left( W^{*}e^{-C}  \right).
$$

\begin{example}
Let $ dim H=1. $ In the weighted Hilbert space $ L^{2}_{\sin\left( \pi t^{\gamma}\right)}, \ \gamma \geq 2, $ 
consider the differential expression
$$
l(u)=u'(t)+a(t)u(t),
$$
where $ a- $real value function and $ a\in C^{(1)}[0,1]. $ Then,  by Theorem \ref{thm1} for formally normality of the minimal operator $ L_{0} $ generated by the differential expression $ l(\cdot) $ in the weighted Hilbert space $ L^{2}_{\sin\left( \pi t^{\gamma}\right)},$ we have
$$
a(t)= c+\frac{\pi\gamma}{2} t^{\gamma -1}\cot\left( \pi t^{\gamma}\right), \ t\in [0,1], \ c>0.
$$
Hence, we get
$$
l(u)=u'(t)+\left( c+\frac{\pi\gamma}{2} t^{\gamma -1}\cot\left( \pi t^{\gamma}\right) \right)u(t), \ c>0.
$$
By Theorem \ref{thm4}, since each unitary operator in $ \mathbb{C} $ can be expressed as $ W=e^{i\varphi}, \ 0\leq \varphi <2\pi $, then each normal extension $ \widetilde{L} $ of the minimal operator $ L_{0} $ is generated by the differential expression $ l(\cdot) $ and the boundary condition
$$
\left( \sqrt{sin\left( \pi t^{\gamma}\right)}u\right) (1)= e^{i\varphi}\left( \sqrt{sin\left( \pi t^{\gamma}\right)}u\right)(0), \ 0\leq \varphi <2\pi.
$$
The $ \varphi $ is uniquely determined by the extension $ \widetilde{L} $, i.e. $  \widetilde{L}=L_{\varphi} $.
By Theorem \ref{thm6}, the spectrum of the normal extension $ L_{\varphi} $ in $ L^{2}_{\sin\left(\pi t^{\gamma}\right)} $ is of the form
$$
\sigma (L_{\varphi})=\left\lbrace \lambda\in \mathbb{C}: \lambda = c+(\varphi+2n\pi)i, \ n\in\mathbb{Z}\right\rbrace , \ 0\leq \varphi <2\pi.
$$
\end{example}

\section{Asymptotical behaviour of $ s- $ numbers of inverse for normal extensions}
\label{Sec:6}

In this section, the discreteness of the spectrum of normal extensions of $ L_{0} $ in $ L^{2}_{\alpha} $ will be examined. Also, the asymptotic behaviour of the singular numbers of the inverse of such normal extensions is explored, as well as the issue of belonging to Schatten von-Neumann classes. \\
\begin{definition} \cite{Gohberg}
Let $ \mathcal{H} $ be a Hilbert space, $ \mathfrak{S}_{\infty}(\mathcal{H}) $ be a class of linear compact operators in $ \mathcal{H} $ and the operator $ \mathcal{T} $ be linear compact operator in $ \mathcal{H} $. The eigenvalues of the operator $ \left(  \mathcal{T}^{*}\mathcal{T}\right) ^{1/2} \in \mathfrak{S}_{\infty}(\mathcal{H}) $ are called the $ s- $numbers of the operator $ \mathcal{T}. $ We shall enumerate the nonzero s-numbers in decreasing order, taking account of their multiplicities, so that
$$
s_{n}(\mathcal{T})=\lambda_{n}(\left(  \mathcal{T}^{*}\mathcal{T}\right) ^{1/2}), \ n=1,2,...
$$
The Schatten-von Neumann operator ideals are defined as
$$
\mathfrak{S}_{p}(\mathcal{H})= \left\lbrace \mathcal{T} \in \mathfrak{S}_{\infty}(\mathcal{H}):\sum \limits _{n=1}^\infty s_{n}^{p}(\mathcal{T})<\infty \right\rbrace , \ 1\leq p < \infty 
$$
\end{definition}
Let note the following simple fact.
\begin{theorem}
\label{thm7}
Let $ dim H<\infty . $ Any normal extension $ L_{W} $ has a pure point spectrum and $ s- $numbers of extensions $ L_{W}^{-1} $ have the same asymptotics  
$$
s_{n}\left( L_{W}^{-1} \right) \sim\dfrac{1}{2n\pi}, \ \text{as} \ n\rightarrow\infty .
$$
\end{theorem}
Let prove the following result.
\begin{theorem}
\label{thm8}
If $ C^{-1} \in \mathfrak{S}_{\infty} (H)$ and the operator $ L_{W} $ is any normal extension of $ L_{0}, $ then $ L_{W}^{-1} \in \mathfrak{S}_{\infty} (L^{2}_{\alpha}).$
\end{theorem}
\noindent {\it Proof.}
The theorem can be proved by modernizing the method in \cite{Ismailov2006}.
\begin{corollary}
Let $ L_{W} $ be  any normal extension of $ L_{0} $ and $ \lambda \in \rho (L_{W}).$ Then $ R_{\lambda}(L_{W})\in \mathfrak{S}_{\infty}(L_{\alpha}^{2}). $ 
\end{corollary}
This result can be obtained from the following relation
$$
R_{\lambda}(L_{W})=L_{W}^{-1}-\lambda R_{\lambda}(L_{W})L_{W}^{-1 }.
$$
With the use of the method in the proof of Theorem \ref{thm8} the following result can be proved.
\begin{corollary}
If $ C^{-1}\in \mathfrak{S}_{p}(H), \ p\geq 1 $  and $ L_{W} $ is any normal extension of $ L_{0} $, then $ L_{W}^{-1}\in \mathfrak{S}_{p}(L_{\alpha}^{2}). $ 
\end{corollary}
Furthermore, from the representation of resolvent $ R_{\lambda}(L_{W}), \ \lambda \in \rho (L_{W}), $ of the operator we have the following corollary.
\begin{corollary}
\label{corollary3}
Let $ L_{W_{1}}, \ L_{W_{2}} $ be two normal extensions of $ L_{0} $ in $ L_{\alpha}^{2} $ and $ \lambda \in \rho(L_{W_{1}})\cap \rho (L_{W_{2}}). $ Then we have 
$$
R_{\lambda}(L_{W_{1}})-R_{\lambda}(L_{W_{2}}) \in \mathfrak{S}_{p}(L_{\alpha}^{2}), \ p\geq 1
$$
if and only if
$$
W_{1}-W_{2}\in \mathfrak{S}_{p}(H), \ p\geq 1.
$$
\end{corollary}
Let prove a result on the structure of the spectrum of the normal extension of $ L_{0}. $
\begin{theorem}
\label{thm9}
If $ C^{-1} \in \mathfrak{S}_{\infty} (H)$ and $ L_{W} $ is any normal extension of $ L_{0} $ in $ L_{\alpha}^{2} $, then the spectrum of $ L_{W} $ is of the form
$$
\sigma(L_{W})=\left\lbrace \lambda_{n}(C)-i\left( arg \lambda_{n}(We^{-C})+2k\pi \right)   , n\in \mathbb{N}; k\in \mathbb{Z}\right\rbrace .
$$
\end{theorem}
\noindent {\it Proof.}
Since $ \sigma(L_{W})=\sigma(M_{W})=\sigma_{p}(M_{W}),  $ then we investigate of the structure of spectrum of $ M_{W} $. By Theorem \ref{thm6}, we get
$$
\sigma (L_{W})=\left\lbrace \lambda\in \mathbb{C}: \lambda = -\left( ln\vert \mu \vert+iarg\mu +2k\pi i\right) , \ k\in \mathbb{Z}, \ \mu\in \sigma\left( W^{*}e^{-C}  \right) \right\rbrace .
$$
Since $ C^{-1} \in \mathfrak{S}_{\infty}(H), $ we have
$
W^{*}e^{-C} =W^{*}\left( C e^{-C}\right) C^{-1}\in \mathfrak{S}_{\infty} (H).
$\\
For any eigenvector $ x_{\lambda}\in H $ corresponding to the eigenvalue 
$
\lambda \in \sigma_{p} \left(  W^{*}e^{-C} \right) 
$
we obtain
$$
W^{*}e^{-C}x_{\lambda}=\lambda \left( W^{*}e^{-C}\right) x_{\lambda} .
$$
Since $ \overline{\lambda}\in \mathbb{C} $ is an eigenvalue of the adjoint operator to $  W^{*}e^{-C } $ which is the operator $ e^{-C} W $ with the same eigenvector $ x_{\lambda} $ in $ H, $ then we have 
\begin{eqnarray}
e^{-C} W W^{*}e^{-C} x_{\lambda} & = &  \lambda \left( W^{*}e^{-C} \right) e^{-C} W x_{\lambda}
 =  \lambda \left( W^{*}e^{-C} \right)\overline{\lambda \left( W^{*}e^{-C}\right)}x_{\lambda}. \nonumber 
\end{eqnarray}
Then
$$
e^{-2C} x_{\lambda}= \vert \lambda \left( W^{*}e^{-C}   \right)\vert^{2}x_{\lambda} .  
$$
Hence we get
\begin{eqnarray}
 \vert \lambda \left( W^{*}e^{-C} \right)\vert^{2}  =  \lambda e^{-2C }
  =  e^{-2\lambda (C)} \nonumber
\end{eqnarray}
which implies
\begin{eqnarray}
 \vert \mu \vert  =  \vert \lambda \left( W^{*} e^{-C }   \right)\vert 
  =  e^{-\lambda (C)}. \nonumber
\end{eqnarray}
Therefore, we obtain
$$
ln \vert \mu \vert=-\lambda (C) .
$$ 
As a result, we get
$$
\sigma(L_{W})=\left\lbrace \lambda \in \mathbb{C}: \lambda=\lambda_{n}(C)-i\left( arg \lambda_{n}(We^{-C})+2k\pi \right)   , n\in \mathbb{N}; k\in \mathbb{Z}\right\rbrace .
$$
Let prove the main theorem of this section.
\begin{theorem}
\label{thm10}
If  $ C^{-1}\in \mathfrak{S}_{\infty}(H) $ and $ \lambda_{n}(C)\sim c n^{\beta}, \ 0<c, \ \beta<\infty,$ then $ L_{W}^{-1}\in \mathfrak{S}_{\infty}(L_{\alpha}^{2}) $ and 
$$
s_{n}(L_{W}^{-1})\sim dn^{-\theta}, \ 0<d<\infty, \ \theta=\dfrac{\beta}{1+\beta}.
$$
\end{theorem}
\noindent {\it Proof.}
Since $ C^{-1}\in \mathfrak{S}_{\infty} (H), $ then by \cite{Ismailov2006}  $ M_{W}^{-1}\in \mathfrak{S}_{\infty}(L^{2}) $ and by Theorem \ref{thm8} $ \ L_{W}^{-1} \in \mathfrak{S}_{\infty}(L_{\alpha}^{2}) $   and $ s_{n}\left( L_{W}^{-1} \right)=s_{n}\left( M_{W}^{-1} \right), \ n\in \mathbb{N}.  $ With a method similar to \cite{Ismailov2006}, we have
$$
s_{m}(M_{W}^{-1})\sim  \vert\lambda_{m}(M_{W}^{-1})\vert  \sim d m^{-\theta} , \ \text{as} \  m\rightarrow\infty, \ 0<d<\infty.
$$.
Hence,
$$
s_{m}(L_{W}^{-1})\sim  \vert\lambda_{m}(L_{W}^{-1})\vert  \sim d m^{-\theta} , \ \text{as} \  m\rightarrow\infty, \ 0<d<\infty.
$$.
\section{One generalization}
\label{Sec:7}
In this section, an approach is developed on all normal extension expressed in the weighted Hilbert spaces. 

\indent Let $ H $ be a separable Hilbert space, $ \alpha, \ \beta : [0,1]\rightarrow [0,\infty ), \ \alpha, \ \beta \in C^{(2)}[0,1],  $
$$
\lambda\left( \left\lbrace t\in [0,1]: \alpha (t)=0\right) \right\rbrace = \lambda\left( \left\lbrace t\in [0,1]: \beta (t)=0\right) \right\rbrace =0, 
$$ 
where $ \lambda $  is Lebesque measure in $ \mathbb{R}. $ Also, let $ L_{\alpha}^{2}:=L_{\alpha}^{2}(H,(0,1)) $ and $ L_{\beta}^{2}:=L_{\beta}^{2}(H,(0,1)) $ be two weighted Hilbert spaces of vector functions on $ (0,1). $

It is noted that the following operator 
$$
\mathcal{F}_{\alpha}^{\beta}(u)=\left( \sqrt{\frac{\beta}{\alpha}} u\right) (t) \ a. \ e. \ t\in (0,1), \ \mathcal{F}_{\alpha}^{\beta}: L_{\beta}^{2}\rightarrow L_{\alpha}^{2}
$$
is a unitary operator.
In this case,
$$
\left( \mathcal{F}_{\alpha}^{\beta}\right) ^{-1}(u)=\mathcal{F}_{\beta}^{\alpha}(u)=\left( \sqrt{\frac{\alpha}{\beta}} u\right) (t) \ a. \ e. \ t\in (0,1), \ \left( \mathcal{F}_{\alpha}^{\beta}\right) ^{-1}: L_{\alpha}^{2}\rightarrow L_{\beta}^{2}.
$$
Hence, from the condition of normality (\ref{equ1}),
\begin{eqnarray}
l(\mathcal{F}_{\alpha}^{\beta}u)(t) & = & (\mathcal{F}_{\alpha}^{\beta}u)'(t)+A(t)(\mathcal{F}_{\alpha}^{\beta}u)(t) \nonumber \\
& = & \left( \sqrt{\frac{\beta}{\alpha}}u\right)'(t)+A(t)\left( \sqrt{\frac{\beta}{\alpha}}u\right)(t) \nonumber \\
& = & \sqrt{\frac{\beta(t)}{\alpha(t)}}\left( u'(t)+\frac{\beta'(t)}{2\beta (t)}u(t)-\frac{\alpha '(t)}{2\alpha (t)}u(t)+A(t)u(t) \right) \nonumber \\
& = & \sqrt{\frac{\beta(t)}{\alpha(t)}}\left( u'(t)+\frac{\beta'(t)}{2\beta (t)}u(t)+\left( A_{r}(t)-\frac{\alpha '(t)}{2\alpha (t)}\right) u(t)+iA_{i}(t)u(t) \right) \nonumber \\
& = & \sqrt{\frac{\beta(t)}{\alpha(t)}}\left( u'(t)+\frac{\beta'(t)}{2\beta (t)}u(t)+Cu(t)+iA_{i}(t)u(t) \right) \ a. \ e. \ t\in (0,1). \nonumber 
\end{eqnarray}
Thus, we get 
$$
\sqrt{\frac{\alpha(t)}{\beta(t)}}l(\mathcal{F}_{\alpha}^{\beta}u)(t)=u'(t)+\frac{\beta'(t)}{2\beta (t)}u(t)+Cu(t)+iA_{i}(t)u(t) \ a. \ e. \ t\in (0,1),
$$
that is,
$$
\left(  \mathcal{F}_{\beta}^{\alpha}l \mathcal{F}_{\alpha}^{\beta}\right) (u)= u'(t)+\frac{\beta'(t)}{2\beta (t)}u(t)+Cu(t)+iA_{i}(t)u(t) \ a. \ e. \ t\in (0,1), \ \mathcal{F}_{\beta}^{\alpha}l \mathcal{F}_{\alpha}^{\beta}: L_{\beta}^{2}\rightarrow L_{\beta}^{2}.
$$
Let define the following linear differential-operator expression in the form
$$
j(u)= u'(t)+\left( \frac{\beta'(t)}{2\beta (t)}I+C\right) u(t)+iA_{i}(t)u(t) \ a. \ e. \ t\in (0,1)
$$
in the weighted Hilbert space $ L_{\beta}^{2} $ of vector functions on $ (0,1). $ 

\indent The minimal operator $ J_{0}(J_{0}^{+}) $ corresponding to differential-operator expression $ j(\cdot) (j^{+}(\cdot )) $ in $ L_{\beta}^{2} $ can be constructed by using same technique \cite{Hormander,Ismailov2006}. The operator $ J(J^{+}) $ adjoint of $ J_{0}^{+}(J_{0}) $ is called the maximal operator corresponding to $ j(\cdot )(j^{+}(\cdot ))  $ in $ L_{\beta}^{2} $ \cite{Hormander,Ismailov2006}. That is, $ J=(J_{0}^{+})^{*}, \ J^{+}=J_{0}^{*} $ and $ J_{0}\subset J, \  J_{0}^{+}\subset J^{+} $. In this case
$
D(J)=W_{2,\beta}^{1}(D(A),(0,1)) \ \text{and} \ D(J_{0})=\stackrel{_o}{W}_{2,\beta}^{1}(D(A),(0,1)).
$ 

\begin{lemma}
For the formal normality of the minimal operator $ L_{0} $ in the weighted Hilbert space $ L_{\alpha}^{2} $ the necessary and sufficient condition is  the formal normality  of the  minimal operator $ J_{0} $ in the weighted Hilbert space $ L_{\beta}^{2} $.
\end{lemma}
\noindent {\it Proof.} Note that
\begin{eqnarray}
& & J_{0}=\mathcal{F}_{\beta}^{\alpha} L_{0} \mathcal{F}_{\alpha}^{\beta}, \ J=\mathcal{F}_{\beta}^{\alpha} L \mathcal{F}_{\alpha}^{\beta} \nonumber \\
& & L_{0}=\mathcal{F}_{\alpha}^{\beta} J_{0} \mathcal{F}_{\beta}^{\alpha}, \ L=\mathcal{F}_{\alpha}^{\beta} J \mathcal{F}_{\beta}^{\alpha}. \nonumber 
\end{eqnarray}
Now, let $ L_{0} $ be formal normal operator in the weighted Hilbert space $ L_{\alpha}^{2} $. Hence, for each $ u\in D(J_{0}) $ we have 
\begin{eqnarray}
\Vert J_{0} u\Vert_{\beta}^{2} & = & \langle J_{0} u,J_{0} u \rangle_{\beta} \nonumber \\
& = & \langle \mathcal{F}_{\beta}^{\alpha} L_{0} \mathcal{F}_{\alpha}^{\beta} u,\mathcal{F}_{\beta}^{\alpha} L_{0} \mathcal{F}_{\alpha}^{\beta} u \rangle_{\beta} \nonumber \\
& = & \langle L_{0} \mathcal{F}_{\alpha}^{\beta} u, L_{0} \mathcal{F}_{\alpha}^{\beta} u \rangle_{\alpha} \nonumber \\
& = & \langle L\mathcal{F}_{\alpha}^{\beta} u, L\mathcal{F}_{\alpha}^{\beta} u \rangle_{\alpha} \nonumber \\
& = & \langle \mathcal{F}_{\alpha}^{\beta} J \mathcal{F}_{\beta}^{\alpha}\mathcal{F}_{\alpha}^{\beta} u, \mathcal{F}_{\alpha}^{\beta} J \mathcal{F}_{\beta}^{\alpha}\mathcal{F}_{\alpha}^{\beta} u \rangle_{\alpha} \nonumber \\
& = & \langle \mathcal{F}_{\alpha}^{\beta} Ju, \mathcal{F}_{\alpha}^{\beta} Ju \rangle_{\alpha} \nonumber \\
& = & \langle  Ju, Ju \rangle_{\beta} \nonumber \\
& = &  \Vert Ju\Vert_{\beta}^{2}. \nonumber
\end{eqnarray}
Consequently, we get
$$
\Vert J_{0} u\Vert_{\beta}^{2}= \Vert Ju\Vert_{\beta}^{2},  \ u\in D(J_{0})\subset D(J),
$$
that is, $ J_{0} $ is formal normal operator in the weighted Hilbert space $ L_{\beta}^{2}$.

On the contrary, if $ J_{0} $ is formal normal operator in the weighted Hilbert space $ L_{\beta}^{2}, $ then by the same method as in the first part of the proof it can be proved that $ L_{0} $ is formal normal operator in the weighted Hilbert space $ L_{\alpha}^{2}. $ \\
Thus, the proof of completed.
\begin{theorem}
\label{thm11}
For the normality of an extension $ \widetilde{L} $ of the minimal operator generated by the differential-operator expression $ l(\cdot) $ in the weighted Hilbert space $ L_{\alpha}^{2} $ the necessary and sufficient condition is  the normality of the extension $ \widetilde{J}=\mathcal{F}_{\beta}^{\alpha} \widetilde{L} \mathcal{F}_{\alpha}^{\beta} $ of the  minimal operator generated by the differential-operator expression $ j(\cdot) $ in the weighted Hilbert space $ L_{\beta}^{2} $.

Moreover, the spectrum sets of $ \widetilde{L} $ and $ \widetilde{J} $ extensions are equal, that is
$$
\sigma (\widetilde{L})= \sigma (\widetilde{J}).
$$
\end{theorem}
\noindent {\it Proof.} It is clear that if the operator $ \widetilde{J} $ is some extension of the minimal operator $ J_{0}, $ that is $ J_{0}\subset \widetilde{J}\subset J, $ then
$$
J_{0}=\mathcal{F}_{\beta}^{\alpha} L_{0} \mathcal{F}_{\alpha}^{\beta}, \ J_{0}\subset \widetilde{J}=\mathcal{F}_{\beta}^{\alpha} \widetilde{L} \mathcal{F}_{\alpha}^{\beta}\subset J, \ J=\mathcal{F}_{\beta}^{\alpha} L \mathcal{F}_{\alpha}^{\beta}.
$$
Let $ \widetilde{L} $ be any normal extension of the minimal operator $ L_{0} $ in $ L_{\alpha}^{2} $. Then, we have
\begin{eqnarray}
\widetilde{J}^{*}\widetilde{J}& = & \left( \mathcal{F}_{\beta}^{\alpha} \widetilde{L}^{*} \mathcal{F}_{\alpha}^{\beta}\right) \left( \mathcal{F}_{\beta}^{\alpha} \widetilde{L} \mathcal{F}_{\alpha}^{\beta}\right)  \nonumber \\
& = & \mathcal{F}_{\beta}^{\alpha} \widetilde{L}^{*}\widetilde{L} \mathcal{F}_{\alpha}^{\beta} \nonumber \\
& = & \mathcal{F}_{\beta}^{\alpha} \widetilde{L}\widetilde{L}^{*} \mathcal{F}_{\alpha}^{\beta} \nonumber \\
& = & \left( \mathcal{F}_{\beta}^{\alpha} \widetilde{L} \mathcal{F}_{\alpha}^{\beta}\right) \left( \mathcal{F}_{\beta}^{\alpha} \widetilde{L}^{*} \mathcal{F}_{\alpha}^{\beta}\right)  \nonumber \\
& = & \widetilde{J}\widetilde{J}^{*}. \nonumber
\end{eqnarray}
Hence, $ \widetilde{J} $ is the normal extension of the minimal operator $ J_{0} $ in $ L_{\beta}^{2}. $

On the contrary, if $ \widetilde{J} $ is a normal extension of the minimal operator $ J_{0} $ in $ L_{\beta}^{2}, $ then by the same method as in the first part of the proof it can be proved that $ \widetilde{L}=\mathcal{F}_{\alpha}^{\beta} \widetilde{J} \mathcal{F}_{\beta}^{\alpha} $ is normal extension of the minimal operator $ L_{0} $ in $ L_{\alpha}^{2}. $

Additionally, using the same technique as in the proof of the Theorem \ref{thm5} it can be shown that $ \sigma (\widetilde{L})= \sigma (\widetilde{J}) $.

\begin{example}
Let $ dim H=1 $, $ \alpha(t)=t^{\gamma}, \ \gamma \geq 2 $ and $ \beta (t)=(1-t)^{\delta}, \ \delta \geq 2 . $ In the weighted Hilbert space $ L^{2}_{t^{\gamma}}, $ consider the following first order differential expression
$$
l(u)=u'(t)+a(t)u(t),
$$
where $ a- $real value function and $ a\in C^{(1)}[0,1]. $ Then,  by Theorem \ref{thm1} for formally normality of the minimal operator $ L_{0} $ generated by the differential expression $ l(\cdot) $ in the weighted Hilbert space $ L^{2}_{t^{\gamma}},$ we have
$$
a(t)=c+\frac{\gamma}{2t}, \ t\in [0,1], \ c>0.
$$
Hence, we get
$$
l(u)=u'(t)+\left( c+\frac{\gamma}{2t}\right)u(t) , \ c>0
$$
In this case, in the weighted Hilbert space $ L^{2}_{(1-t)^{\delta}},$ we have the following differential operator expression
$$
\left( \frac{t^{\frac{\gamma}{2}}}{(1-t)^{\frac{\delta}{2}}}l\frac{(1-t)^{\frac{\delta}{2}}}{t^{\frac{\gamma}{2}}}\right) (u)=j(u)=u'(t)+\left( c+\frac{\delta}{2(1-t)}\right)u(t) , \ c>0.
$$
By Theorem \ref{thm4}, each normal extension $ \widetilde{J} $ of the minimal operator $ J_{0} $ is generated by the differential expression $ j(\cdot) $ and the boundary condition
\begin{equation}
\label{equ5}
\left( (1-t)^{\frac{\delta}{2}}u\right) (1)= e^{i \varphi}\left( (1-t)^{\frac{\delta}{2}}u\right)(0),
\end{equation}
for some $ \varphi : 0\leq \varphi <2\pi. $
The $ \varphi $ is uniquely determined by the extension $ \widetilde{J} $, i.e. $  \widetilde{J}=J_{\varphi} $.
By Theorem \ref{thm6}, the spectrum of the normal extension $ J_{\varphi} $ in $ L^{2}_{(1-t)^{\delta}} $ is of the form
$$
\sigma (J_{\varphi})=\left\lbrace \lambda\in \mathbb{C}: \lambda = c+(\varphi+2n\pi)i, \ n\in\mathbb{Z}\right\rbrace , \ 0\leq \varphi <2\pi.
$$

On the other hand, by Theorem \ref{thm11} for the normality of an extension $ L_{\varphi}=\frac{(1-t)^{\frac{\delta}{2}}}{t^{\frac{\gamma}{2}}}J_{\varphi}\frac{t^{\frac{\gamma}{2}}}{(1-t)^{\frac{\delta}{2}}}$ of the minimal operator $ L_{0} $ generated by the differential expression $ l(\cdot) $ in the weighted Hilbert space $ L_{t^{\gamma}}^{2} $ the necessary and sufficient condition is  the normality of the extension $ J_{\varphi} $ of the  minimal operator generated by the differential expression $ j(\cdot) $ in the weighted Hilbert space $ L_{(1-t)^{\delta}}^{2} $. \\
Hence, if $ u\in D(L_{\varphi}), $ then $ \frac{t^{\frac{\gamma}{2}}}{(1-t)^{\frac{\delta}{2}}}u\in D(J_{\varphi}), $ that is, by the boundary condition (\ref{equ5})
$$
\left( (1-t)^{\frac{\delta}{2}}\frac{t^{\frac{\gamma}{2}}}{(1-t)^{\frac{\delta}{2}}}u\right) (1)= e^{i \varphi}\left( (1-t)^{\frac{\delta}{2}}\frac{t^{\frac{\gamma}{2}}}{(1-t)^{\frac{\delta}{2}}}u\right)(0),
$$
for some $ \varphi : 0\leq \varphi <2\pi. $ Thus, we have the normal extension $ L_{\varphi} $ of the minimal operator $ L_{0} $ is generated by the differential expression $ l(\cdot) $ and the boundary condition
$$
\left( t^{\frac{\gamma}{2}}u\right) (1)= e^{i \varphi}\left( t^{\frac{\gamma}{2}}u\right)(0),
$$
for some $ \varphi : 0\leq \varphi <2\pi. $

Moreover, by Theorem \ref{thm11} the spectrum sets of $ L_{\varphi} $ and $ J_\varphi $ extensions are equal, that is
$$
\sigma (L_{\varphi})= \sigma (J_{\varphi}).
$$
\end{example}

\section*{Data Availability Statement}  No data was used for the research described in
the article.

\section*{Declarations} 
\textbf{Conflict of interest} 
The authors declare that they have no conflict of interest.


\noindent
\begin{thebibliography}{35}

\footnotesize

\bibitem{Akhiezer} N. I. Akhiezer, I. M. Glazman, Theory of Linear Operators in Hilbert Space, \emph{Frederick Ungar Publishing Co., New York: Ungar}, (1963).

\bibitem{Biriuk} G. Biriuk, E. A. Coddington, {Normal extensions of unbounded formally normal operators}, \emph{Journal of Mathematics and Mechanics}, \textbf{13}, 617-638, (1964).

\bibitem{Biyarov1993} B. N. Biyarov, M. Otelbaev, {Description of normal extensions}, \emph{Matematicheskie Zametki}, \textbf{53(5)}, 21-28, (1993).

\bibitem{Biyarov2016} B. N. Biyarov, {Normal extensions of linear operators}, \emph{Eurasian Mathematical Journal}, \textbf{7(3)}, 17-32, (2016).

\bibitem{Bruk} V. M. Bruk, {Some problems of the spectral theory of the linear differential equation for first order with unbounded operator coefficient}, \emph{Funktsional'nyi Analiz i ego Prilozheniya}, \textbf{1}, 15-25, (1973), (in Russian).

\bibitem{Coddington} E. A. Coddington, Extension Theory of Formally Normal and Symmetric Subspaces, \emph{American Mathematical Society, Rhode Island}, (1973).

\bibitem{Davis} R. H. Davis, {Singular Normal Differential Operators}, \emph{PhD thesis, University of California}, (1955).

\bibitem{Everitt1997} W. N. Everitt, L. Markus, {The Glazman-Krein-Naimark theorem for ordinary differential operators}, \emph{Operator Theory: Advances and Applications}, \textbf{98}, 118-130, (1997).

\bibitem{Everitt} W. N. Everitt, A. Poulkou, {Some observations and remarks on differential operators generated by first-order boundary value problems}, \emph{Journal of Computational and Applied Mathematics}, \textbf{153}, 201-2011, (2003).

\bibitem{Goldstein} J. A. Goldstein, Semigroups of Linear Operators and Applications, \emph{Courier Dover Publications, New York}, (2017).

\bibitem{Gohberg} I. C. Gohberg, M. G. Krein, Introduction to the Theory of Linear Non-Selfadjoint Operators in Hilbert Space, \emph{Nauka, Moscow}, (1965).

\bibitem{Gorbachuk} V. I. Gorbachuk, M. L. Gorbachuk, \emph{Boundary Value Problems for Operator Differential Equations}, \emph{Kluwer Academic Publishers, Dordrecht}, (1991).

\bibitem{Hao} X. Hao, M. Zhang, J. Sun, A. Zettl, {Characterization of domains of self-adjoint ordinary differential operators of any order even or odd}, \emph{Electronic Journal of Qualitative Theory of Differential Equations}, \textbf{61}, 1-19, (2017).

\bibitem{Hormander} L. H\"ormander {On the theory of general partial differential operators}, \emph{Acta Mathematica}, \textbf{94}, 161-248, (1955).

\bibitem{Ismailov2003} Z. I. Ismailov, {On the normality of first order differential operators}, \emph{Bulletin of the Polish Academy of Sciences}, \textbf{51 (2)}, 39-45, (2003).

\bibitem{Ismailov2005} Z. I. Ismailov, {On the coefficients of normal differential operators of higher order}, \emph{Transactions of National Academy of Sciences of Azerbaijan. Series of Physical-Technical and Mathematical Sciences}, \textbf{25(4)}, 55-62, (2005).

\bibitem{Ismailov2006} Z. I. Ismailov, {Compact inverses of first-order normal differential operators}, \emph{Journal of Mathematical Analysis and Applications}, \textbf{320(1)}, 266-278, (2006).

\bibitem{Ismailov2012a} Z. I. Ismailov, M. Erol, {Normal differential operators of third-order}, \emph{Hacettepe Journal of Mathematics and Statistic}, \textbf{41(5)}, 675-688, (2012).

\bibitem{Ismailov2012b} Z. I. Ismailov, M. Erol, {Normal differential operators of first-order with smooth coefficients}, \emph{Rocky Mountain Journal of Mathematics},\textbf{42(2)}, 1100-1110, (2012).

\bibitem{Ismailov2011} Z. I. Ismailov, E. Otkun Çevik, E. Unluyol, {Compact inverses of multipoint normal differential operators for first order}, \emph{Electronic Journal of Differential Equations}, \textbf{89}, 1-11, (2011).

\bibitem{Ismailov2012c} Z. I. Ismailov, R. Öztürk Mert, {Normal extensions of a singular multipoint differential operator of first order}, \emph{Electronic Journal of Differential Equations}, \textbf{36}, 1-9, (2012).

\bibitem{Ismailov2014a}  Z. I. Ismailov, R. Öztürk Mert, {Normal extensions of a singular differential operator on the right semi-axis}, \emph{Eurasian Mathematical Journal},\textbf{5(3)}, 117-124, (2014).

\bibitem{Ismailov2014b}  Z. I. Ismailov, M. Sertbaş, B. O. Güler, {Normal extensions a first order differential operators}, \emph{Filomat}, \textbf{28(5)}, 917-923, (2014).

\bibitem{Kato} T. Kato, Perturbation Theory for Linear Operators, \emph{Springer-Verlag, New York}, (1966).

\bibitem{Kilpi1} Y. Kilpi, {Über lineare normale transformationen in Hilbertschen raum}, \emph{Annales Academiae Scientiarum Fennicae. Series A. I. Mathematica-Physica}, \textbf{154}, 1-38, (1953) (in German).

\bibitem{Kilpi2} Y. Kilpi, {Über das komplexe momenten problem}, \emph{Annales Academiæ Scientiarum Fennicæ, Mathematica}, \textbf{236}, 1-32, (1957)(in German).

\bibitem{Kilpi3} Y. Kilpi, {Über die Anzahl der hypermaximalen normalen Fortsetzungen normaler Transformationen}, \emph{Annales Universitatis Turkuensis, Series A. I. Astronomica-Chemica-Physica-Mathematica}, \textbf{65}, 1-12, (1963) (in German).
       
\bibitem{Krein} S. G. Krein, Linear Differential Equations in Banach Space, \emph{American Mathematical Society, Providence R.I}, (1971).

\bibitem{Naimark} M. A. Naimark, Linear Differential Operators, \emph{Frederick Ungar Publishing Co., New York}, (1968).

\bibitem{Rofe-Beketov} F. S. Rofe-Beketov, A. M. Kholkin, Spectral Analysis of Differential Operators, World Scientific Monograph Series in Mathematics, \emph{World Scientific Publishing Co. Pte Ltd., New Jersey}, (2005).

\bibitem{Shilov} G. E. Shilov, Mathematical Analysis: The Second Special Course, \emph{Nauka, Moscow}, (1965) (in Russian).

\bibitem{Neumann} J. von Neumann, Allgemeine Eigenwerttheorie Hermitescher Funktionaloperatoren, \emph{Math. Ann.}, \textbf{102}, 49-131, (1929-1930) (in German).

\bibitem{Zettl} A. Zettl and J. Sun, {Survey article: self-adjoint ordinary differential operators and their spectrum}, \emph{Rocky Mountain Journal of Mathematics}, \textbf{45(3)}, 763-886 (2015).



        
\end{thebibliography}
\end{document}